\documentclass{amsart}

\usepackage{latexsym,amssymb}
\usepackage{amsmath,amsthm}

\usepackage{amsfonts}
\usepackage{graphicx}   
\usepackage{url}

\usepackage{mathtools}
\mathtoolsset{showonlyrefs}

\def\R{\mathbb{R}}
\def\N{\mathbb{N}}

\def\ve{\varepsilon}

\def\supp{\mathop{\text{\normalfont supp}}}

\newtheorem{thm}{Theorem}[section]
\newtheorem{lem}[thm]{Lemma}
\newtheorem{cor}[thm]{Corollary}
\newtheorem{prop}[thm]{Proposition}

\theoremstyle{definition}
\newtheorem{definition}[thm]{Definition}

\theoremstyle{remark}
\newtheorem{rem}[thm]{Remark}

\numberwithin{equation}{section}

\begin{document}

\title[Fractional rearrangement and obstacle problems]{Optimal rearrangement problem and  normalized obstacle problem in the fractional setting}

\author[J.F. Bonder, Z. Cheng and H. Mikayelyan]{Juli\'an Fern\'andez Bonder,  Zhiwei Cheng and Hayk Mikayelyan}

\address[J.F. Bonder]{Departamento de Matem\'atica FCEN - Universidad de Buenos Aires and IMAS - CONICET. Ciudad Universitaria, Pabell\'on I (C1428EGA) Av. Cantilo 2160. Buenos Aires, Argentina.}

\email{jfbonder@dm.uba.ar}

\urladdr{http://mate.dm.uba.ar/~jfbonder}

\address[Z. Cheng and H. Mikayelyan]{Mathematical Sciences, University of Nottingham Ningbo, 199 Taikang East Road, Ningbo 315100, PR China.}

\email[H. Mikayelyan]{Hayk.Mikayelyan@nottingham.edu.cn}

	

\begin{abstract}
We consider an optimal rearrangement minimization problem involving the fractional Laplace operator $(-\Delta)^s$, $0<s<1$, and Gagliardo seminorm $|u|_s$. We prove the existence of the unique minimizer, analyze its properties as well as derive the non-local and highly non-linear PDE it satisfies
$$
-(-\Delta)^s U-\chi_{\{U\leq 0\}}\min\{-(-\Delta)^s U^+;1\}=\chi_{\{U>0\}},
$$ 
which happens to be the fractional analogue of the normalized obstacle problem $\Delta u=\chi_{\{u>0\}}$. 
\end{abstract}
\maketitle



\section{Introduction}

One of the classical problems in rearrangement theory is the minimization of the functional
\begin{equation}\label{Dir-functional}
\Phi(f)=\int_D |\nabla u_f|^2 dx,
\end{equation}
where $u_f$ is the unique solution of the Dirichlet boundary value problem
\begin{equation}\label{main}
\begin{cases} -\Delta u_f = f  & \text{in } D,
 \\
 u_f=0  & \text{on } \partial D,  \end{cases}
\end{equation}
and $f$ belongs to the set
$$
\bar{\mathcal{R}}_\beta=\left\{f\in L^\infty(D)\colon 0\leq f \leq 1,\,\,\int_D fdx=\beta \right\}.
$$
Recall that $\bar{\mathcal{R}}_\beta$ is the closure in the weak* topology of the rearrangement class
$$
\mathcal{R}_\beta := \left\{f\in L^\infty(D)\colon f=\chi_E,\,\, |E|=\beta \right\}.
$$
This minimization problem is related to the stationary heat equation 
$$
\underbrace{\partial_t u}_{=0}-\Delta u = f
$$
in the domain $D$, which is under the action of the external heat source modeled by the force function $f$. The Dirichlet boundary condition, $u=0$ on $\partial D$, models the constant temperature on the boundary of $ D$. Different force functions $f$ result different heat distributions $u_f$. The minimizer $\hat{f}$ of the functional  \eqref{Dir-functional} is the force function from a certain rearrangement class $\mathcal{R}$, which is resulting the most uniformly distributed heat $u_{\hat{f}}$.

The problem and its variations, such as the $p-$harmonic case, has been studied by several authors (see \cite{B1, B2, BM, EL, Kbook}), and the results, for this particular setting, can be formulated in the following theorem.

 \begin{thm}\label{thm1}
 There exists a unique solution $\hat{f}\in\mathcal{R}_\beta$ of the minimization problem \eqref{Dir-functional}. For the function
 $\hat{u}=u_{\hat{f}}$ there exists a constant $\alpha>0$ such that
 \vspace{2mm}

\hspace{3cm}$\bullet$ $0< \hat{u}\leq\alpha$ in $ D$,
  \vspace{1mm}

\hspace{3cm}$\bullet$ $\hat{f}=\chi_{\{\hat{u}<\alpha\}}$,
   \vspace{1mm}

\hspace{3cm}$\bullet$ $\hat{u}= \alpha$ in $\{\hat{f}=0\}$.
\vspace{2mm}

\noindent Moreover, the function $U=\alpha-\hat{u}$ is the minimizer of the functional
$$
J(w)= \int_ D |\nabla w|^2+ 2w^+\, dx,
$$
among functions $w\in H^1( D)$ with boundary values $\alpha$ on $\partial  D$, and
solves the obstacle problem equation
\begin{equation}\label{classical-obstacle}
\Delta U=\chi_{\{U>0\}} \quad \text{ in } D.
\end{equation}
\end{thm}

We refrain from presenting here details about the obstacle problem \eqref{classical-obstacle}, which is one of the classical free boundary problems (see \cite{CaffRevisited}).

In recent years there has been a great development of nonlocal diffusion problems, mainly due to some interesting new applications to different fields of the natural sciences such as some physical models \cite{DiNezza-Palatucci-Valdinoci, Eringen, Giacomin-Lebowitz, Laskin, Metzler-Klafter, Zhou-Du}, finance \cite{Akgiray-Booth, Levendorski, Schoutens}, fluid dynamics \cite{Constantin}, ecology \cite{Humphries, Massaccesi-Valdinoci, Reynolds-Rhodes} and image processing \cite{Gilboa-Osher}.

Among these models for nonlocal diffusion, probably the most important one is given by the fractional laplacian $(-\Delta)^s$, $(0<s<1)$ that is given (for smooth functions) as
\begin{align*}
(-\Delta)^s u(x) &:= \text{p.v. } \int_{\R^n} \frac{u(x)-u(y)}{|x-y|^{n+2s}}\, dy\\
&= \lim_{\ve\downarrow 0} \int_{\R^n\setminus B_\ve(x)} \frac{u(x)-u(y)}{|x-y|^{n+2s}}\, dy.
\end{align*}

This operator is given as the gradient of the nonlocal Gagliardo energy
\begin{equation}\label{Gagliardo.energy}
|u|_s^2 := \iint_{\R^{2n}} \frac{|u(x)-u(y)|^2}{|x-y|^{n+2s}}\, dxdy,
\end{equation}
that is the nonlocal analog of the Dirichlet energy $\|\nabla u\|_2^2$.

In view of the increasing interest in analyzing nonlocal diffusion models, it naturally comes into attention considering problem \eqref{main} where the Laplace operator is replace by its fractional counterpart.

Therefore, in this paper, similar to the way it has been done in \cite{M}, we will consider an optimal rearrangement problem and derive a related free boundary problem.

More precisely, we consider the minimization problem
$$
\Phi_s(f)\to \text{min},
$$
where $\Phi_s(f) = |u_f|_s^2$, $u_f$ is the unique solution to
$$
(-\Delta)^s u_f = f \quad\text{ in }D\quad \text{ and } \quad u=0 \quad \text{ in } D^c
$$
and $f\in \bar{\mathcal{R}}_\beta$.

We show existence and uniqueness of a solution to the fractional rearrangement optimization problem and show that if $\hat f$ is the solution and $\hat u = u_{\hat f}$, then $0\le \hat u\le \alpha$ for some $\alpha>0$ and, moreover, $\hat U  = \alpha-\hat u$ is the unique solution to the normalized fractional obstacle problem
$$
\chi_{\{U>0\}}\le -(-\Delta)^s U\le \chi_{\{U\ge 0\}} \quad \text{ in } D \quad\text{ and }\quad U=\alpha \text{ in } D^c.
$$
Also, we analyze the behavior of such solutions as the fractional parameter $s$ goes to 1.

Finally, we show that the solution to the fractional normalized obstacle problem is also the solution to the (highly nonlinear) equation
$$
-(-\Delta)^s U-\chi_{\{U\leq 0\}}\min\{-(-\Delta)^s U^+;1\}=\chi_{\{U>0\}},
$$ 
in $D$ with $U=\alpha$ in $D^c$.

\subsection*{Organization of the paper}
 In Section \ref{sec-pre} we give a brief introduction to fractional calculus, in Section 3 we analyze the optimal rearrangement problem in the fractional setting and show its relation with the normalized fractional obstacle problem. In Section 4, we study the behavior of the optimal fractional rearrangement problem as $s\to 1$. Finally, in Section 5, we further analyze the normalized fractional obstacle problem and derive a (highly) nonlinear equation that the solution satisfies. 

\section{Preliminaries}\label{sec-pre}

\subsection{A very short tour through the basics of the fractional Laplacian}

All the results in this section are either well-known or easily proved, so we just recall them for further references without any attempt of giving proofs. 

The fractional order Sobolev spaces $H^s(\R^n)$ (for $0<s<1$) is defined as
$$
H^s(\R^n)=\{v\in L^2(\R^n)\colon |v|^2_s<\infty\},
$$
where $|\cdot|_s$ is the Gagliardo energy given by \eqref{Gagliardo.energy}. This space is a Hilbert space with inner product given by
$$
(u, v)_{H^s(\R^n)} = \int_{\R^n} u(x)v(x)\, dx + \iint_{\R^{2n}} \frac{(u(x)-u(y))(v(x)-v(y))}{|x-y|^{n+2s}}\, dxdy.
$$
For a brief summary of the properties of fractional order Sobolev spaces $H^s$, we refer to the survey article \cite{DiNezza-Palatucci-Valdinoci}.

Further we denote by $H^{-s}(\R^n)$ the topological dual space of $H^s(\R^n)$ and for a domain $D\subset \R^n$, we denote
$$
H^s_0(D)=\{v\in H^s(\R^n)\colon  v=0 \text{ a.e. in } D^c \}.
$$ 
Recall that for Lipschitz domains $D$, the space $H^s_0(D)$ coincides with the closure of test functions with compact support inside $D$. We will also denote by $H^{-s}(D)$ the topological dual space of $H^s_0(D)$. Observe that we have
$$
H^s_0(D)\subset H^s(\R^n)\subset L^2(\R^n)\subset H^{-s}(\R^n)\subset H^{-s}(D),
$$
with continuous inclusions. Moreover, since $\mathcal D \subset H^s(\R^n)$ with a dense inclusion, then $H^{-s}(\R^n)\subset \mathcal D'$ and, if $D$ is Lipschitz, then $H^{-s}(D)\subset \mathcal D'(D)$.

Recall that if $D$ is bounded, the following Poincar\'e type inequality holds true
\begin{equation}\label{poincare}
\|u\|_2\le C |u|_s \quad \text{ for all } u\in H^s_0(D).
\end{equation}

An easy fact is that the Gagliardo semi-norm $|\cdot|_s^2$ is G\^ateaux - differentiable in $H^s(\R^n)$ and 
\begin{equation}\label{GN-der}
\lim_{\ve\to 0}\ve^{-1}(|u+\ve v|^2_s - |u|^2_s) =
2\iint_{\R^{2n}}\frac{(u(x)-u(y))(v(x)-v(y))}{|x-y|^{n+2s}}\, dxdy,
\end{equation}
for every $u, v\in H^s(\R^n)$.

Furthermore, for a function $u\in H^s(\R^n)$ we can also define the fractional Laplace operator as
\begin{equation}\label{lim.def}
(-\Delta)^s u(x)=p.v. \int_{\R^n}\frac{u(x)-u(y)}{|x-y|^{n+2s}}\, dy=
\lim_{\ve\to 0}(-\Delta)^s_\ve u(x),
\end{equation}
where
$$
(-\Delta)^s_\ve u(x)=\int_{\R^n\setminus B_\ve(x)}\frac{u(x)-u(y)}{|x-y|^{n+2s}}\, dy
$$
and the limit is understood in $H^{-s}(\R^n)$.

Moreover, it holds that
$$
\langle (-\Delta)^s u, v\rangle = \frac12\iint_{\R^{2n}}\frac{(u(x)-u(y))(v(x)-v(y))}{|x-y|^{n+2s}}dxdy \leq  \frac12 |u|_s |v|_s,
$$
for any $u, v\in H^s(\R^n)$.

For any $f\in H^{-s}(D)$ we say $u_f\in H^s_0(D)$ solves the fractional boundary value problem in $D$ with homogeneous Dirichlet boundary condition
\begin{equation}\label{main-fr}
\begin{cases} (-\Delta)^s u_f = f   & \mbox{in } D,
 \\
 u_f=0  & \mbox{in }  D^c,  \end{cases}
\end{equation}
if the equation is satisfied in the sense of distributions. Equivalently, if
\begin{equation}\label{partint-fr}
\frac{1}{2}\iint_{\R^{2n}}\frac{(u_f(x)-u_f(y))(v(x)-v(y))}{|x-y|^{n+2s}}\, dxdy = \int_D fv\, dx
\end{equation}
for any $v\in H^s_0(D)$. It is easily seen from Riesz representation Theorem, using Poincar\'e inequality \eqref{poincare}, that for any $f\in H^{-s}(D)$ there exists a unique $u_f\in H^s_0(D)$ satisfying \eqref{partint-fr}.

To finish these preliminaries we refer the reader to \cite{Silv}, and recall that for $f\in L^\infty(D)$ the weak solution of \eqref{main-fr}, $u_f\in C^{0,\delta}_{loc}(D)$ for some $\delta>0$, if $s\leq\frac{1}{2}$ and $u_f\in C^{1,\delta}_{loc}(D)$ for some $\delta>0$, if $s> \frac{1}{2}$. Moreover, $u_f$ is a {\em strong} solution to \eqref{main-fr}, namely the limit in \eqref{lim.def} exists pointwise a.e. and the equation \eqref{main-fr} is also satisfied pointwise a.e.


\section{The optimal fractional rearrangement problem}\label{opt.rear}

Let us now introduce the fractional analogue of the optimal rearrangement problem given in Theorem \ref{thm1}. Given $f\in L^2(D)$, let $u_f$ be the solution of \eqref{main-fr} and let us define the functional
\begin{equation}\label{Phis}
\Phi_s(f) = |u_f|_s^2.
\end{equation}
We are going to consider the minimization of the functional $\Phi_s$ on the closed, convex set $\bar{\mathcal{R}}_\beta$, for $0<\beta<|D|$. The main result of this section is the following theorem.
\begin{thm}\label{thm.opt}
There exists a unique minimizer $\hat{f}\in \bar{\mathcal{R}}_\beta\setminus \mathcal{R}_\beta$ such that 
$$
\Phi_s(\hat{f})\leq \Phi_s(f)
$$
for any $f\in  \bar{\mathcal{R}}_\beta$. Moreover, for some $\alpha>0$ the function $\hat{u}=u_{\hat{f}}$
satisfies the following conditions
$$
\{ \hat{f} < 1\}\subset\{\hat{u}=\alpha\},\quad \{\hat{u}<\alpha\}\subset\{\hat{f}=1\},\quad \hat f>0 \quad \text{and}\quad 0\le\hat u\le \alpha \text{ in }D.
$$
\end{thm}

\begin{rem}
Observe that this result shows a remarkable difference with the local optimal rearrangement problem, since the optimal configuration $\hat f$ for the fractional case is not a characteristic function. c.f. Theorem \ref{thm1}.
\end{rem}

For the proof of Theorem \ref{thm.opt} we need a couple of lemmas.

\begin{lem}\label{ext-r}
The set $\bar{\mathcal{R}}_\beta\subset L^\infty(D)$ is convex and
$$
\text{ext}(\bar{\mathcal{R}}_\beta)=\mathcal{R}_\beta,
$$
where for a convex set $C$, $\text{ext}(C)$ denotes the extreme points of $C$.
\end{lem}

\begin{proof}
The proof is standard and is ommitted.
\end{proof}

\begin{lem}
Let $\Phi_s$ be the functional defined in \eqref{Phis}. Then $\Phi_s\colon \bar{\mathcal{R}}_\beta\to\R$ is strictly convex and sequentially lower semi-continuous with respect to the weak* topology. Moreover, there exists a unique minimizer $\hat{f}$ of the functional $\Phi_s$ in $\bar{\mathcal{R}}_\beta$.
\end{lem}

\begin{proof}
The strict convexity is a direct consequence of the linearity $u_{f_1+f_2} = u_{f_1} + u_{f_2}$ and the strict convexity of $t\mapsto t^2$. Moreover, from \eqref{partint-fr}, H\"older's inequality and \eqref{poincare}, we obtain
$$
|u_f|_s\le C\|f\|_2.
$$
Therefore, $f\mapsto u_f$ is strongly continuous from $L^2(D)$ into $H^s_0(D)$ and hence, $\Phi_s$ is strongly continuous from $L^2(D)$ into $\R$. Since $\Phi_s$ is convex, it follows that is sequentially weakly lower semicontinuous. 

Finally, observe that if $f_n\in \bar{\mathcal{R}}_\beta$ is such that $f_n\stackrel{*}{\rightharpoonup} f$ weakly* in $L^\infty(D)$, then $f_n\rightharpoonup f$ weakly in $L^2(D)$, and so
$$
\Phi_s(f)\le \liminf\Phi_s(f_n).
$$

To finish the proof just notice that the existence of a minimizer follows from Banach-Alaoglu's theorem and the uniqueness of the minimizer from the strict convexity of $\Phi_s$ and the convexity of $\bar{\mathcal{R}}_\beta$.
\end{proof}

Now we are ready to prove Theorem \ref{thm.opt}.
\begin{proof}[Proof of Theorem \ref{thm.opt}]
The proof will be divided into a series of claims.
\medskip

{\bf Claim 1.}
$$
\int_D \hat{u}\hat{f}\, dx \le \int_D \hat{u}f\, dx\quad \text{ for any } f\in \bar{\mathcal{R}}_\beta.
$$

Let us take $\Psi\colon L^2(D)\to \bar\R$ defined as
$$
\Psi(f)=\Phi_s(f)+\xi_{\bar{\mathcal{R}}_\beta}(f),
$$
where $\xi_{\bar{\mathcal{R}}_\beta}(f)$ is the indicator function, i.e. 
$$
\xi_{\bar{\mathcal{R}}_\beta}(f)=\begin{cases} 
0, & \text{if } f\in \bar{\mathcal{R}}_\beta,\\
\infty,  & \text{if }   f\notin \bar{\mathcal{R}}_\beta. 
\end{cases}.
$$
Observe that $\Psi$ is strictly convex there. Moreover, it is easy to see that $\hat{f}$ minimizes $\Psi$ in $L^2(D)$. Thus 
$$
0\in \partial \Psi(\hat{f}),
$$
where
$$
\partial \Psi(\hat{f}) = \left\{g\in L^2(D)\colon \Psi(f)-\Psi(\hat{f})\ge \int_D g(f-\hat{f})\, dx, \text{ for any } f\in  L^2(D)  \right\}
$$
is the sub-differential of $\Psi$ at $\hat{f}$.

From \eqref{partint-fr} and \eqref{GN-der} we get that
$$
\partial \Phi_s(\hat{f})=\{2 \hat{u}\}.
$$
Moreover
\begin{align*}
\partial \xi_{\bar{\mathcal{R}}_\beta}(\hat{f}) &= \left\{g\in L^2(D)\colon \xi_{\bar{\mathcal{R}}_\beta}(f)-\xi_{\bar{\mathcal{R}}_\beta}(\hat{f}) \ge \int_D g(f-\hat{f})\, dx, \text{ for any } f\in  L^2(D)  \right\}\\
&= \left\{g\in L^2(D)\colon 0 \ge \int_D g(f-\hat{f})\, dx, \text{ for any } f\in  \bar{\mathcal{R}}_\beta  \right\}.
\end{align*}
Therefore the equation
$$
0\in \partial \Psi(\hat{f})=\partial \Phi_s(\hat{f})+\partial \xi_{\bar{\mathcal{R}}_\beta}(\hat{f}),
$$
implies that 
$$
-\hat{u}\in \partial \xi_{\bar{\mathcal{R}}_\beta}(\hat{f})
$$
and thus the claim.
\medskip

{\bf Claim 2.} There exists a function $\tilde{f}=\chi_E\in \mathcal{R}_\beta$ such that 
$$
\int_D \hat{u}\tilde{f}dx\leq
\int_D \hat{u}fdx
$$
for any $ f\in \bar{\mathcal{R}}_\beta  $.

This follows from Claim 1, Lemma \ref{ext-r}, and the fact that the minimum of the linear functional $L(f)=\int_D \hat{u}f\, dx$ on a bounded closed convex set $\bar{\mathcal{R}}_\beta$ is attained in an extreme point $\tilde{f}=\chi_E\in \mathcal{R}_\beta$.
\medskip

{\bf Claim 3.} There exists  $\alpha>0$ such that
$$
\{\hat{u}<\alpha\} \subset E\subset \{\hat{u}\leq \alpha\}.
$$

The proof is an immediate consequence of the bathtub principle for $\tilde{f}$. See \cite[Theorem 1.14]{LLbook}.
\medskip

{\bf Claim 4.}  
$$
\hat{f}=1\text{ in } \{\hat{u}<\alpha\}.
$$

The proof is again an immediate consequence of the bathtub principle for $\hat{f}$.
\medskip

{\bf Claim 5.}
$$
\{\hat{u}>\alpha\}\subset \{\hat{f}=0\}.
$$

Since $\hat f, \tilde f\in \bar{\mathcal{R}}_\beta$, we have that
\begin{align*}
\beta &= \int_D \hat f\, dx = \int_{\{\hat u <\alpha\}} \hat f\, dx + \int_{\{\hat u =\alpha\}} \hat f\, dx + \int_{\{\hat u >\alpha\}} \hat f\, dx\\
&=\int_D \tilde f\, dx = \int_{\{\hat u <\alpha\}} \tilde f\, dx + \int_{\{\hat u =\alpha\}} \tilde f\, dx + \int_{\{\hat u >\alpha\}} \tilde f\, dx
\end{align*}
Therefore, by Claims 3 and 4, we obtain that
\begin{equation}\label{obs1}
   \int_{\{\hat u =\alpha\}} \hat f\, dx + \int_{\{\hat u > \alpha\}} \hat f\, dx= \int_{\{\hat u =\alpha\}} \tilde f\, dx.
\end{equation}
On the other hand, by Claims 1 and 2, we get
$$
\int_D \hat u\hat f\, dx = \int_D \hat u \tilde f\, dx
$$
that together with \eqref{obs1} give us
\begin{align*}
\alpha\int_{\{\hat{u} =\alpha \}}\tilde{f}\, dx &= \alpha\int_{\{\hat{u} =\alpha \}} \hat{f}\, dx + \alpha \int_{\{\hat{u} >\alpha \}}\hat{f}\, dx\\
& \le \int_{\{\hat{u} =\alpha \}}\hat{u}\hat{f}\, dx + \int_{\{\hat{u} >\alpha \}}\hat{u}\hat{f}\, dx\\
&= \int_{\{\hat u=\alpha\}} \hat u\tilde f\, dx \\
&= \alpha\int_{\{\hat{u} =\alpha \}}\tilde{f}\, dx,
\end{align*}
which implies 
$$
\alpha \int_{\{\hat{u} >\alpha \}}\hat{f}\, dx = \int_{\{\hat{u} >\alpha \}}\hat{u}\hat{f}\, dx,
$$ 
and thus the claim.
\medskip

{\bf Claim 6.}
$$
\{\hat{u}>\alpha\}= \emptyset.
$$

For $\beta>\alpha$ let us take $\phi(x)=(\hat u(x)-\beta)^+$. Since $\supp \phi=\omega \subset\{\hat{u}>\alpha\},$ claim 5 implies that 
\begin{align*}
0=2\langle (-\Delta)^s \hat{u},\phi\rangle= & \iint_{\R^{2n}}\frac{(\hat{u}(x)-\hat{u}(y))(\phi(x)-\phi(y))}{|x-y|^{n+2s}}\, dxdy\\
=&\underbrace{\int_\omega\int_\omega \frac{(\hat{u}(x)-\hat{u}(y))^2}{|x-y|^{n+2s}}\, dxdy}_{\geq 0}\\
&+\underbrace{\int_\omega\left(\int_{\R^n\setminus\omega}\frac{(\hat{u}(x)-\hat{u}(y))(\hat{u}(x)-\beta)}{|x-y|^{n+2s}}\, dy\right)\, dx}_{\geq 0}\\
&+\underbrace{\int_{\R^n\setminus\omega}\left(\int_\omega\frac{(\hat{u}(x)-\hat{u}(y))(\beta-\hat{u}(y))}{|x-y|^{n+2s}}\, dy\right)\, dx}_{\geq 0}\\
&+\underbrace{\int_{\R^n\setminus\omega}\left(\int_{\R^n\setminus\omega}\frac{(\hat{u}(x)-\hat{u}(y))(0-0)}{|x-y|^{n+2s}}\, dy\right)\, dx}_{= 0}.
\end{align*}
Thus, $|\omega|=|\{\hat{u}>\beta\}|=0$ for any $\beta>\alpha$. Moreover, since $\hat u\in C^{0,\delta}_{loc}(D)$ for some $\delta>0$, it is easy to see that the claim follows.
\medskip

{\bf Claim 7.}
$$
|\{\hat{f}=0\}|=0.
$$

Since $(-\Delta)^s \hat{u}=\hat{f}\in L^\infty(D)$ and $\hat{f}\geq 0$ it is enough to check $\hat{f}>0$ point-wise. 

Taken Claim 4 we need to check this only in the set $\{\hat{u}=\alpha\}$. But
\begin{multline*}
\hat{f}(x)=p.v.\int_{\R^{n}}\frac{\hat{u}(x)-\hat{u}(y)}{|x-y|^{n+2s}} dy=\\
\lim_{\epsilon\to 0} \int_{\R^{n}\setminus B_\epsilon(x)}\frac{\hat{u}(x)-\hat{u}(y)}{|x-y|^{n+2s}} dy
>\int_{\R^{n}\setminus D}\frac{\alpha}{|x-y|^{n+2s}} dy>0.
\end{multline*}
This proves the claim.

\medskip

The proof of the Theorem is complete.
\end{proof}

\section{The behavior of the optimal rearrangement problem as $s\to 1$}

In this section we analyze the behavior of the optimal fractional rearrangement problem as the fractional parameter $s$ goes to 1. For that purpose, we need to consider here the normalizing constant $C(n,s)$ that is defined as
$$
C(n,s)=\left(\int_{\R^n} \frac{1-\cos(\zeta_1)}{|\zeta|^{n+2s}}\, d\zeta\right)^{-1}
$$
and we need to modify the definitions of the fractional laplacian and of the Gagliardo seminorm accordingly, namely, we consider
$$
|u|_s^2 = C(n,s) \iint_{\R^{2n}} \frac{|u(x)-u(y)|^2}{|x-y|^{n+2s}}\, dxdy
$$
and
$$
(-\Delta)^s u (x) = \text{p.v. } C(n,s)\int_{\R^n} \frac{u(x)-u(y)}{|x-y|^{n+2s}}\, dy.
$$
It is a well known fact that this normalizing constant behaves like $(1-s)$ for $s$ close to 1. Moreover, the following result holds
\begin{prop}\label{fixed.u} 
Let $u\in L^2(\R^n)$ be fixed. Then we have that
$$
|u|_s^2\to \|\nabla u\|_2^2 \quad \text{and}\quad (-\Delta)^s u\to -\Delta u\qquad \text{as } s\to 1.
$$
where the first limit is understood as a limit if $u\in H^1(\R^n)$ and as $\liminf |u|_s^2 = \infty$ if $u\not\in H^1(\R^n)$ and the second limit is in the sense of distributions.
\end{prop} 
For a proof, see for instance \cite{DiNezza-Palatucci-Valdinoci} and \cite{BBM}.

Moreover, it is shown in \cite{BBM} the following stronger statement.
\begin{prop}\label{sequence.u}
Given a sequence $s_k\to 1$ and $\{u_k\}_{k\in\N}\subset L^2(\R^n)$ such that
$$
\sup_{k\in \N} \|u_k\|_2 <\infty\quad \text{and}\quad \sup_{k\in\N} |u_k|_{s_k} <\infty,
$$
then there exists a function $u\in H^1(\R^n)$ such that (up to a subsequence), 
$$
u_k\to u \text{ strongly in } L^2_{loc}(\R^n) \quad \text{and}\quad \|\nabla u\|_2^2\le \liminf_{k\to\infty} |u_k|_{s_k}^2.
$$
\end{prop}

Throughout this section, we will denote by $\hat f_s$ the optimal load for $\Phi_s$, $\hat u_s = u_{\hat f_s}$ the solution to \eqref{main-fr}. Also, denote $\Phi(f)$ as
$$
\Phi(f) = \int_D |\nabla u_f|^2\, dx,
$$
where in this section, $u_f$ will denote the solution to
$$
\begin{cases}
-\Delta u_f = f & \text{in } D\\
u=0 & \text{on }\partial D.
\end{cases}
$$
Finally, denote by $\hat f\in \bar{\mathcal R}_\beta$ the solution to the minimization problem
$$
\Phi(\hat f) = \inf_{f\in \bar{\mathcal R}_\beta} \Phi(f).
$$
So the main result in this section is the following:
\begin{thm}\label{sto1}
Under the above notations, $\hat f_s\stackrel{*}{\rightharpoonup}\hat f$ weakly* in $L^\infty$ as $s\to 1$. Moreover we also obtain that
$$
\Phi_s(\hat f_s)\to \Phi(\hat f)\quad\text{and}\quad \hat u_s\to \hat u\text{ strongly in } L^2(D),
$$
as $s\to 1$.
\end{thm}

For the proof of Theorem \ref{sto1} we need the concept of $\Gamma-$convergece. This concept was introduced by De Giorgi in the 60s and is now a well understood tool to deal with the convergence of minimum problems. For a throughout introduction to the subject, we cite \cite{DalMaso}. Let us recall now the definition of $\Gamma-$convergence and some of its properties.
\begin{definition}
Let $X$ be a metric space and $F_n, F\colon X\to\bar\R$. We say that $F_n$ $\Gamma-$converges to $F$, and is denoted by $F_n\stackrel{\Gamma}{\to} F$, is the following two inequalities hold true
\begin{itemize}
\item ($\liminf-$inequality) For any $x\in X$ and any sequence $\{x_n\}_{n\in\N}\subset X$ such that $x_n\to x$ in $X$, it holds that
$$
F(x)\le \liminf_{n\to\infty} F_n(x_n).
$$
\item ($\limsup-$inequality) For any $x\in X$, there exists a sequence $\{y_n\}_{n\in\N}\subset X$ such that $y_n\to x$ in $X$ and
$$
F(x)\ge \limsup_{n\to\infty} F_n(y_n).
$$
\end{itemize}
\end{definition}

The main feature of the $\Gamma-$convergence is that it implies the convergence of minima. In fact we have the following:
\begin{thm}\label{thm.gamma}
Let $X$ be a metric space and $F_n, F\colon X\to\bar R$ be functions such that $F_n\stackrel{\Gamma}{\to}F$. Moreover, assume that for each $n\in\N$, there exists $x_n\in X$ such that
$$
F_n(x_n)=\inf_X F_n
$$
and that $\{x_n\}_{n\in\N}\subset X$ is precompact. Then
$$
\lim_{n\to\infty} \inf_X F_n = \inf_X F
$$
and every accumulation point of the sequence $\{x_n\}_{n\in\N}$ is a minimum point of $F$.
\end{thm}

The proof of Theorem \ref{thm.gamma} is easy and can be found in \cite{DalMaso}.

The following result is key in the proof of Theorem \ref{sto1}.
\begin{thm}\label{key.thm}
Given $0<s<1$, let $f_s\in L^2(\Omega)$ be such that $f_s\rightharpoonup f$ weakly in $L^2(\Omega)$ and let $u_s\in H^s_0(\Omega)$ and $u\in H^1_0(\Omega)$ be the solutions to 
$$
(-\Delta)^s u_s = f_s \quad \text{in }\Omega,\quad u_s=0 \quad \text{in }\R^n\setminus \Omega
$$
and
$$
-\Delta u = f \quad \text{in }\Omega,\quad u = 0 \quad \text{on }\partial \Omega
$$
respectively.

Then $u_s\to u$ strongly in $L^2(\Omega)$. Moreover,
$$
|u_s|_s \to \|\nabla u\|_2.
$$
\end{thm}

\begin{proof}
Let $F_s, F\colon L^2(\R^n)\to \bar \R$ given by
$$
F_s(v) := \begin{cases}
\frac12 |v|_s^2 - \int_\Omega f_s v\, dx & \text{if } v\in H^s_0(\Omega),\\
+\infty & \text{otherwise}
\end{cases}
$$
and
$$
F(v) := \begin{cases}
\frac12 \|\nabla v\|_2^2 - \int_\Omega f v\, dx & \text{if } v\in H^1_0(\Omega),\\
+\infty & \text{otherwise}.
\end{cases}
$$

Since $\int_\Omega f_s v_s\, dx \to \int_\Omega fv\, dx$ if $v_s\to v$ strongly in $L^2(\R^n)$, from Propositions \ref{fixed.u} and \ref{sequence.u} we can conclude that $F_s\stackrel{\Gamma}{\to} F$ as $s\to 1$.

Now, observe that 
$$
F_s(u_s) = \inf_{v\in L^2(\R^n)} F_s(v) \quad \text{and} \quad F(u) = \inf_{v\in L^2(\R^n)} F(v).
$$

The trivial estimate $|u_s|_s\le \|f_s\|_2$ imply that, for any $s_k\to 1$, the sequence $\{u_{s_k}\}_{k\in\N}\subset L^2(\R^n)$ is precompact. Then from Theorem \ref{thm.gamma} we obtain that $u_s\to u$ strongly in $L^2(\Omega)$.

Finally, 
$$
\lim |u_s|_s^2 = \lim \int_\Omega f_s u_s\, dx = \int_\Omega fu\, dx = \|\nabla u\|_2^2.
$$
This completes the proof.
\end{proof}

Now we are ready to prove the main result of the section.
\begin{proof}[Proof of Theorem \ref{sto1}]
Let $\hat f_s\in \bar{\mathcal R}_\beta$ be the optimal load for $\Phi_s$. Observe that, for a subsequence, $\hat f_s\stackrel{*}{\rightharpoonup} f$ weakly* in $L^\infty(\Omega)$ for some $f\in \bar{\mathcal R}_\beta$. Moreover, this convergence also holds weakly in $L^2(\Omega)$.

From Theorem \ref{key.thm} we have that $\hat u_s\to u_f$ strongly in $L^2(\Omega)$ and using Theorem \ref{sequence.u} we get
$$
\inf_{\bar{\mathcal R}_\beta} \Phi\le \Phi(f) = \|\nabla u_f\|_s^2\le \liminf |\hat u_s|_s^2 =\liminf \inf_{\bar{\mathcal R}_\beta} \Phi_s.
$$

On the other hand, let $\hat f\in \bar{\mathcal R}_\beta$ be the optimal load for $\Phi$. Then, using the final part of Theorem \ref{key.thm}, we obtain
$$
\limsup \inf_{\bar{\mathcal R}_\beta} \Phi_s\le \lim \Phi_s(\hat f) = \Phi(\hat f) = \inf_{\bar{\mathcal R}_\beta} \Phi.
$$
The proof is complete.
\end{proof}

\section{The normalized fractional obstacle problem}

This section is devoted to the study of the connection between the solutions to the optimal fractional rearrangement problem consider in Section \ref{opt.rear} with solutions of the normalized fractional obstacle problem.

The fractional analogue of the classical obstacle problem has been well known in the literature, however its so called normalized version, i.e., the equation
\begin{equation}\label{norm.obst}
\Delta u=\chi_{\{u>0\}},
\end{equation}
has not been considered. Here we find the corresponding fractional analog of \eqref{norm.obst} and prove that the solution of the fractional rearrangement problem is a solution of the fractional normalized obstacle problem.

Our first result is the following.
\begin{thm}\label{frac.obst}
Let $\hat f\in \bar{\mathcal{R}}_\beta$ be the solution to the optimal fractional rearrangement problem and $\hat u := u_{\hat f}\in H^s_0(D)$ be given by \eqref{main-fr}. Let $\alpha>0$ be the constant given in Theorem \ref{thm.opt}. Then the function $\hat U := \alpha-\hat u$ minimizes the functional
\begin{equation}\label{func-fr}
J(v) = |v|_s^2 + \int_D v^+\, dx
\end{equation}
over the set $H_\alpha = \{v\in H^s_{loc}(\R^n)\colon v-\alpha\in H^s_0(D)\}$. Moreover, $\hat U$ verifies the inequalities
\begin{equation}\label{ineq.obst}
\chi_{\{U>0\}}\le -(-\Delta)^s U\le \chi_{\{U\ge 0\}} \quad \text{in } D
\end{equation}
in the sense of distributions.

Finally, the minimizer of $J$ in $H_\alpha$ is unique and is the unique solution to the inequality \eqref{ineq.obst}.
\end{thm}

\begin{proof}
Let 
$$
I(v)=|v|_s^2 + \int_D \hat{f}v\, dx
$$
and observe that, since $0\le\hat f\le 1$, for any $v\in H_\alpha$ it follows that $J(v)\ge I(v)$.

Next, observe that $I(\hat U)=J(\hat U)$ and so the set of inequalities
$$
J(v)\geq I(v)\geq I(\hat{U})=J(\hat{U}),\quad \text{for any } v\in H_\alpha
$$
imply the desired result.

Next, observe that the inequalities
$$
\chi_{\{\hat{U}>0\}}\leq -(-\Delta)^s \hat{U}\leq\chi_{\{\hat{U}\geq 0\}}.
$$ 
are the Euler-Lagrange equation for the functional $J$ based on the variation $u_\ve(x)=u(x)+\ve \phi(x)$, with $\phi\in C_c^\infty(D)$.

Now, the uniqueness of minimizer for $J$ is an immediate consequence of the strict convexity of $J$.

Assume that the function $U$ satisfies the inequalities \eqref{ineq.obst}, but the unique minimizer of the convex functional $J$ is the function $V\not=U$.

Since $J$ is strictly convex and $J(V)<J(U)$ by taking $U_\ve=U+\ve(V-U)$ we will obtain
$$
\lim_{\ve\to 0^+}\frac{J(U_\ve)-J(U)}{\ve}<0.
$$ 
Thus for $\psi=V-U$ 
\begin{align*}
0>&\lim_{\ve\to 0^+}\frac{J(U_\ve)-J(U)}{\ve}\\
=&\iint_{\R^{2n}}\frac{(u(x)-u(y))(\psi(x)-\psi(y))}{|x-y|^{n+2s}}dxdy+\int_D \chi_{\{U>0\}}\psi+\chi_{\{U=0\}}\psi^+dx\\
=&\underbrace{\iint_{\R^{2n}}\frac{(u(x)-u(y))(\psi^+(x)-\psi^+(y))}{|x-y|^{n+2s}}dxdy+\int_D \chi_{\{U\geq 0\}}\psi^+ dx}_{\geq 0}\\
&-\Big(\underbrace{\iint_{\R^{2n}}\frac{(u(x)-u(y))(\psi^-(x)-\psi^-(y))}{|x-y|^{n+2s}}dxdy+\int_D \chi_{\{U>0\}}\psi^-dx}_{\leq 0}\Big)\geq 0,
\end{align*}
where the last inequality follows from \eqref{ineq.obst}. This is a contradiction and the result follows.
\end{proof}

\begin{rem}
This result again shows an interesting difference between the classical obstacle problem and the fractional normalized version. Observe that in the positivity set, we still have $-(-\Delta)^s \hat U = 1$, but in the zero set the function $\hat U$ is not $s-$harmonic (even if it is identically zero!). The free boundary condition on $\partial\{\hat U>0\}$ is given by the fact that $(-\Delta)^s \hat U$ is a function bounded by $0$ and $1$ across the free boundary.
\end{rem}

The results in Theorem \ref{frac.obst} are not completely satisfactory, since we don't obtain an equation satisfied by $\hat U$ but the inequalities \eqref{ineq.obst}.

Our last result shows that in fact $\hat U$ is the solution to a fully nonlinear equation.

\begin{thm}\label{thm.final}
Let $\hat U$ be solution of the normalized fractional obstacle problem given by Theorem \ref{frac.obst}. Then $\hat U$ is a solution to
\begin{equation}\label{eq.fbp-fr}
\begin{cases} 
-(-\Delta)^s {U}-\chi_{\{{U}\leq 0\}}\min\{-(-\Delta)^s{U}^+;1\} = \chi_{\{{U}>0\}} ,
  & \mbox{in } D,
 \\
 U = \alpha  & \mbox{in }  D^c. \end{cases}
\end{equation}
Moreover, problem \eqref{eq.fbp-fr} is equivalent to \eqref{ineq.obst}. Finally, $U$ verifies \eqref{eq.fbp-fr} if and only if it is a minimizer of $J$ in $H_\alpha$, where $J$ and $H_\alpha$ are given in Theorem \ref{frac.obst}.
\end{thm}

Before we start with the proof, let us observe that for $u\in H^s(\R^n)$,
\begin{equation}\label{pm-ineq}
|u^\pm|_s \leq |u|_s
\end{equation}
and hence $(-\Delta)^s u^\pm\in H^{-s}(\R^n)$. On the other hand $(-\Delta)^s u^+$ is a distribution 
and the expression 
$$
\min\{-(-\Delta)^s u^+;1)=-\max\{(-\Delta)^s u^+,-1\} = 1 - ((-\Delta)^su^+ + 1)^+
$$
makes in general no sense, unless $(-\Delta)^s u^+$ is a signed measure in $D$. Let us further observe that since
$$
\chi_{\{u\le 0\}}(-\Delta)^s u^+ \le 0,
$$
we need to search for solutions of \eqref{eq.fbp-fr} only among functions $u$, such that $(-\Delta)^s u \le 0$ in $D$. This leads us to the introduction of fractional subharmonic functions in $D$, which form a convex subset of $H^s(D)$
\begin{equation}\label{defeq-sub}
H^s_{sub}(D)=\left\{u\in H^s_{loc}(\R^n)\colon (-\Delta)^s u\le 0 \text{ in } D\right\}.
\end{equation}
Here the inequality $(-\Delta)^s u\le 0$ should be understood in the sense of distributions.

The following lemma is essential for the equation \eqref{eq.fbp-fr} to make sense.
\begin{lem}\label{lem-ms}
Let $u\in H^s_{sub}(D)$. Then $u^+\in H^s_{sub}(D)$.
\end{lem}

\begin{proof}
If $u$ is smooth, then the fractional laplacian has pointwise values. In this case, we simply compute:
\begin{itemize}
\item For $x\in \{u\le 0\}$,
$$
(-\Delta)^s u^+(x) = \text{p.v.} \int_{\R^n} \frac{-u^+(y)}{|x-y|^{n+2s}}\, dy\le 0.
$$
\item For $x\in \{u>0\}$,
\begin{align*}
(-\Delta)^s u^+(x) &= \text{p.v.} \int_{\R^n} \frac{u(x) - u^+(y)}{|x-y|^{n+2s}}\, dy\\
&= \text{p.v.} \int_{\R^n} \frac{u(x)-u(y)}{|x-y|^{n+2s}}\, dy - \text{p.v.} \int_{\R^n} \frac{u^-(y)}{|x-y|^{n+2s}}\, dy\\
&\le 0.
\end{align*}
\end{itemize}

For a general $u\in H^s(\R^n)$, we take $\{\rho_\ve\}_{\ve>0}$ a smooth family of approximations of the identity such that $\rho_\ve(z)=\rho_\ve(-z)$ and define $u_\ve = u\ast\rho_\ve$. 

The result of the lemma will follow from the identity
\begin{equation}\label{key.identity}
\langle (-\Delta)^s u_\ve, \phi\rangle = \langle (-\Delta)^s u, \phi_\ve\rangle, 
\end{equation}
for every $\phi\in C^\infty_c(\R^n)$.

Indeed, assuming \eqref{key.identity}, if $(-\Delta)^s u\le 0$, then $(-\Delta)^s u_\ve\le 0$ for every $\ve>0$. Hence, from the smooth case we conclude that $(-\Delta)^s u_\ve^+\le 0$ and since $u_\ve^+\to u^+$ in $H^s$ the result is proved.

It remains to prove \eqref{key.identity}. For that purpose, it is useful to introduce the notation
$$
D^s u(x,y) = \frac{u(x)-u(y)}{|x-y|^{\frac{n}{2}+s}},
$$
the H\"older quotient of order $s$ of $u$. Then, using that $\rho_\ve(-z)=\rho_\ve(z)$, we observe that
\begin{align*}
\langle (-\Delta)^2 u_\ve, \phi\rangle &= \iint_{\R^{2n}} D^s u_\ve(x,y) D^s\phi(x.y)\, dxdy\\
&= \iint_{\R^{2n}} \int_{\R^n} D^s u(x-z, y-z) \rho_\ve(z) D^s\phi(x,y)\, dz\, dxdy\\
&= \iint_{\R^{2n}} \int_{\R^n} D^s u(x, y) \rho_\ve(z) D^s\phi(x+z,y+z)\, dz\, dxdy\\
&= \iint_{\R^{2n}} \int_{\R^n} D^s u(x, y) \rho_\ve(z) D^s\phi(x-z,y-z)\, dz\, dxdy\\
&= \iint_{\R^{2n}} D^s u(x,y) D^s\phi_\ve(x,y)\, dxdy\\
&= \langle (-\Delta)^2 u, \phi_\ve\rangle.
\end{align*}
The proof is now complete.
\end{proof}

\begin{cor}\label{cor-ms}
If $u\in H^s_{sub}(D) $ then $\min\{-(-\Delta)^s u^+;1\}\in L^\infty(D)$.
\end{cor}

Corollary \ref{cor-ms} allows us to formulate the following normalized fractional obstacle problem: 

{\em
For $\alpha>0$ solve 
\begin{equation}\label{fbp-fr-new}
-(-\Delta)^s U - \chi_{\{ U\le 0\}} \min\{-(-\Delta)^s U^+;1\} = \chi_{\{U>0\}} ,
\end{equation}
among functions $U\in H^s(\R^n)$, such that $U=\alpha$ in $D^c$ and $U\in H^s_{sub}(D)$.
}

\noindent The weak formulation of the equation \eqref{fbp-fr-new} is
\begin{align*}
-\frac12 \iint_{\R^{2n}}&\frac{(U(x)-U(y))(\phi(x)-\phi(y))}{|x-y|^{n+2s}}\, dxdy = \\
&\int_D \Big[\chi_{\{{U}\leq 0\}}(x)\min\{-(-\Delta)^s{U}^+;1\}+\chi_{\{{U}>0\}}(x)\Big]\phi(x)\, dx,
\end{align*}
for any $\phi\in H^s_0(D)$.

Now we are ready to prove Theorem \ref{thm.final}.
\begin{proof}[Proof of Theorem \ref{thm.final}] 
We only need to show that problems \eqref{ineq.obst} and \eqref{fbp-fr-new} are equivalent

For convenience let us break down the proof into several claims.
\medskip

{\bf Claim 1.} Assume that $U\in H^s_{loc}(\R^n)$ is a solution of \eqref{fbp-fr-new}. Then $U\geq 0$.

Observe first the following general fact:
$$
(U(x)-U(y))(U^-(x)-U^-(y)) = (U^+(x)-U^+(y))(U^-(x)-U^-(y)) - (U^-(x)-U^-(y))^2.
$$
This simple identity implies that
\begin{equation}\label{identity}
\langle (-\Delta)^s U, U^-\rangle = \langle (-\Delta)^s U^+, U^-\rangle - |U^-|_s^2.
\end{equation}

Now, let us take $\phi=U^-$ as a test function in the weak formulation of \eqref{fbp-fr-new}. Then we obtain
\begin{equation}\label{equation}
\begin{split}
\langle (-\Delta)^s U,U^-\rangle &=-\int_D [\chi_{\{U\leq 0\}}\min(-(-\Delta)^sU^+;1)+\chi_{\{U>0\}} ]U^- \, dx\\
&= -\int_D \min\{-(-\Delta)^sU^+;1\}U^- \, dx\\
&=\int_D\max\{(-\Delta)^2 U^+; -1\} U^-\, dx
\end{split}
\end{equation}
But
\begin{equation}\label{third}
\int_D\max\{(-\Delta)^s U^+; -1\} U^-\, dx\ge \langle(-\Delta)^s U^+, U^-\rangle.
\end{equation}
Therefore, combining \eqref{identity}, \eqref{equation} and \eqref{third}, we arrive at
$$
|U^-|_s^2 \le 0,
$$
and so the claim is proved.
\medskip

{\bf Claim 2.} \eqref{fbp-fr-new} implies \eqref{ineq.obst}. 

It is immediate from Claim 1. 
\medskip

{\bf Claim 3.} \eqref{ineq.obst} implies $U\geq 0$.

The argument is similar to the one of Claim 6 in Theorem  \ref{thm.opt}. Let $U\in H^s_{loc}(\R^n)$ be a solution to \eqref{ineq.obst}. Take $\beta<0$ and $\phi=(U-\beta)^-$, so $\omega = \supp \phi \subset\{U<0\}$. Then

\begin{align*}
0=&2\langle (-\Delta)^s U,\phi\rangle=\iint_{\R^{2n}}\frac{(U(x)-U(y))(\phi(x)-\phi(y))}{|x-y|^{n+2s}}dxdy\\
=&\underbrace{\int_\omega\int_{\omega}\frac{(U(x)-U(y))(U(y)-U(x))}{|x-y|^{n+2s}}dxdy}_{\leq 0}+ \underbrace{\int_\omega\int_{\R^\setminus\omega}\frac{(U(x)-U(y))(\beta-U(x))}{|x-y|^{n+2s}}dxdy}_{\leq 0}\\
&+\underbrace{\int_{\R^\setminus\omega}\int_\omega\frac{(U(x)-U(y))(U(y)-\beta)}{|x-y|^{n+2s}}dxdy}_{\leq 0}+ \underbrace{\int_{\R^\setminus\omega}\int_{\R^\setminus\omega}\frac{(U(x)-U(y))(0-0)}{|x-y|^{n+2s}}dxdy}_{= 0}\\
\le& 0.
\end{align*}
Thus, $|\omega| = |\{U<\beta\}| = 0$ for any $\beta<0$.
\medskip

{\bf Claim 4.} \eqref{ineq.obst} implies \eqref{fbp-fr-new}.

Can be verified directly.
\end{proof}

\subsection*{Acknowledgment}
The research of Zhiwei Cheng and Hayk Mikayelyan has been partly supported by the National Science Foundation of China (grant no.1161101064). Juli\'an F. Bonder is supported by by grants UBACyT 20020130100283BA, CONICET PIP 11220150100032CO and ANPCyT PICT 2012-0153. 

This research was done while J. F. Bonder was a visiting Professor at University of Nottingham at Ningbo, China (UNNC). He wants to thank the institution for the support, the atmosphere and the hospitality that make the stay so enjoyable.

\bibliographystyle{plain}
\bibliography{rearr}

\begin{thebibliography}{10}

\bibitem{Akgiray-Booth}
Vedat Akgiray and G.~Geoffrey Booth.
\newblock The siable-law model of stock returns.
\newblock {\em Journal of Business \& Economic Statistics}, 6(1):51--57, 1988.

\bibitem{BBM}
Jean Bourgain, Haim Brezis, and Petru Mironescu.
\newblock Another look at sobolev spaces.
\newblock In {\em in Optimal Control and Partial Differential Equations}, pages
  439--455, 2001.

\bibitem{B1}
G.~R. Burton.
\newblock Rearrangements of functions, maximization of convex functionals, and
  vortex rings.
\newblock {\em Math. Ann.}, 276(2):225--253, 1987.

\bibitem{B2}
G.~R. Burton.
\newblock Variational problems on classes of rearrangements and multiple
  configurations for steady vortices.
\newblock {\em Ann. Inst. H. Poincar\'e Anal. Non Lin\'eaire}, 6(4):295--319,
  1989.

\bibitem{BM}
G.~R. Burton and J.~B. McLeod.
\newblock Maximisation and minimisation on classes of rearrangements.
\newblock {\em Proc. Roy. Soc. Edinburgh Sect. A}, 119(3-4):287--300, 1991.

\bibitem{CaffRevisited}
L.~A. Caffarelli.
\newblock The obstacle problem revisited.
\newblock {\em J. Fourier Anal. Appl.}, 4(4-5):383--402, 1998.

\bibitem{Constantin}
Peter Constantin.
\newblock Euler equations, {N}avier-{S}tokes equations and turbulence.
\newblock In {\em Mathematical foundation of turbulent viscous flows}, volume
  1871 of {\em Lecture Notes in Math.}, pages 1--43. Springer, Berlin, 2006.

\bibitem{DalMaso}
Gianni Dal~Maso.
\newblock {\em An introduction to {$\Gamma$}-convergence}, volume~8 of {\em
  Progress in Nonlinear Differential Equations and their Applications}.
\newblock Birkh\"auser Boston, Inc., Boston, MA, 1993.

\bibitem{DiNezza-Palatucci-Valdinoci}
Eleonora Di~Nezza, Giampiero Palatucci, and Enrico Valdinoci.
\newblock Hitchhiker's guide to the fractional {S}obolev spaces.
\newblock {\em Bull. Sci. Math.}, 136(5):521--573, 2012.

\bibitem{EL}
Behrouz Emamizadeh and Yichen Liu.
\newblock Constrained and unconstrained rearrangement minimization problems
  related to the {$p$}-{L}aplace operator.
\newblock {\em Israel J. Math.}, 206(1):281--298, 2015.

\bibitem{Eringen}
A.~Cemal Eringen.
\newblock {\em Nonlocal continuum field theories}.
\newblock Springer-Verlag, New York, 2002.

\bibitem{Giacomin-Lebowitz}
Giambattista Giacomin and Joel~L. Lebowitz.
\newblock Phase segregation dynamics in particle systems with long range
  interactions. {I}. {M}acroscopic limits.
\newblock {\em J. Statist. Phys.}, 87(1-2):37--61, 1997.

\bibitem{Gilboa-Osher}
Guy Gilboa and Stanley Osher.
\newblock Nonlocal operators with applications to image processing.
\newblock {\em Multiscale Model. Simul.}, 7(3):1005--1028, 2008.

\bibitem{Humphries}
Nicolas et~al. Humphries.
\newblock Environmental context explains l\'evy and brownian movement patterns
  of marine predators.
\newblock {\em Nature}, 465:1066--1069, 2010.

\bibitem{Kbook}
Bernhard Kawohl.
\newblock {\em Rearrangements and convexity of level sets in {PDE}}, volume
  1150 of {\em Lecture Notes in Mathematics}.
\newblock Springer-Verlag, Berlin, 1985.

\bibitem{Laskin}
Nikolai Laskin.
\newblock Fractional quantum mechanics and {L}\'evy path integrals.
\newblock {\em Phys. Lett. A}, 268(4-6):298--305, 2000.

\bibitem{Levendorski}
Sergei Levendorski.
\newblock Pricing of the american put under l{\'e}vy processes.
\newblock {\em International Journal of Theoretical and Applied Finance},
  7(03):303--335, 2004.

\bibitem{LLbook}
Elliott~H. Lieb and Michael Loss.
\newblock {\em Analysis}, volume~14 of {\em Graduate Studies in Mathematics}.
\newblock American Mathematical Society, Providence, RI, second edition, 2001.

\bibitem{Massaccesi-Valdinoci}
A.~{Massaccesi} and E.~{Valdinoci}.
\newblock Is a nonlocal diffusion strategy convenient for biological
  populations in competition?
\newblock {\em ArXiv e-prints}, March 2015.

\bibitem{Metzler-Klafter}
Ralf Metzler and Joseph Klafter.
\newblock The random walk's guide to anomalous diffusion: a fractional dynamics
  approach.
\newblock {\em Phys. Rep.}, 339(1):77, 2000.

\bibitem{M}
Hayk Mikayelyan.
\newblock Cylindrical optimal rearrangement problem leading to a new type
  obstacle problem.
\newblock {\em ESAIM Control Optim. Calc. Var.}, 24(2):859--872, 2018.

\bibitem{Reynolds-Rhodes}
A.~M. Reynolds and C.~J. Rhodes.
\newblock The lŽvy flight paradigm: random search patterns and mechanisms.
\newblock {\em Ecology}, 90(4):877--887, 2009.

\bibitem{Schoutens}
Wim Schoutens.
\newblock {\em L\'evy Processes in Finance: Pricing Financial Derivatives}.
\newblock Willey Series in Probability and Statistics. Willey, New York, 2003.

\bibitem{Silv}
Luis Silvestre.
\newblock Regularity of the obstacle problem for a fractional power of the
  laplace operator.
\newblock {\em Communications on Pure and Applied Mathematics}, 60(1):67--112,
  2007.

\bibitem{Zhou-Du}
Kun Zhou and Qiang Du.
\newblock Mathematical and numerical analysis of linear peridynamic models with
  nonlocal boundary conditions.
\newblock {\em SIAM J. Numer. Anal.}, 48(5):1759--1780, 2010.

\end{thebibliography}

\end{document}